\documentclass{amsart}

\usepackage{amsmath,amsfonts,amsthm,amssymb,verbatim,enumerate,quotes,graphicx}
\usepackage[flushmargin]{footmisc}
\usepackage[noadjust]{cite}
\usepackage{color}
\newcommand{\comments}[1]{}
\newtheorem{theo}{Theorem}
\newtheorem{lemma}{Lemma}[section]
\newtheorem{corollary}{Corollary}[section]
\newtheorem{example}{Example}

\numberwithin{equation}{section}
\def \isnatural {\in\mathbb{N}}

\def \iscomplex {\in\mathbb{C}}

\def\rom{\textup}
\newcommand{\intb} {\partial_{\text{\rmfamily\rom{int}\normalfont}}}
\newcommand{\tef}{transcendental entire function}
\newcommand\qfor{\quad\text{for }}

\makeatletter
\def\blfootnote{\xdef\@thefnmark{}\@footnotetext}
\makeatother
\begin{document}
%
% OR THIS
%
%\title[A partition of the fast escaping set]{A partition of the fast escaping set of a transcendental entire function}
\title[Maximally and non-maximally fast escaping points]{Maximally and non-maximally fast escaping points of transcendental entire functions}
\author{D. J. Sixsmith}
\address{Department of Mathematics and Statistics \\
	 The Open University \\
   Walton Hall\\
   Milton Keynes MK7 6AA\\
   UK}
\email{david.sixsmith@open.ac.uk}
%%%%%%%%%%%%%
%
% ABSTRACT
%
%%%%%%%%%%%%%
\begin{abstract}
We partition the fast escaping set of a {\tef} into two subsets, the maximally fast escaping set and the non-maximally fast escaping set. These sets are shown to have strong dynamical properties. We show that the intersection of the Julia set with the non-maximally fast escaping set is never empty. The proof uses a new covering result for annuli, which is of wider interest.

It was shown by Rippon and Stallard that the fast escaping set has no bounded components. In contrast, by studying a function considered by Hardy, we give an example of a {\tef} for which the maximally and non-maximally fast escaping sets each have uncountably many singleton components.
\end{abstract}
\maketitle
%
%%%%%%%%%%%%%
%
% INTRO
%
%%%%%%%%%%%%%
\blfootnote{2010 \itshape Mathematics Subject Classification. \normalfont Primary 37F10; Secondary 30D05.}
\blfootnote{The author was supported by Engineering and Physical Sciences Research Council grant EP/J022160/1.}
\section{Introduction}
Suppose that $f:\mathbb{C}\rightarrow\mathbb{C}$ is a {\tef}. 
The \itshape Fatou set \normalfont $F(f)$ is defined as the set $z\iscomplex$ such that $\{f^n\}_{n\isnatural}$ is a normal family in a neighbourhood of $z$. The \itshape Julia set \normalfont $J(f)$ is the complement in $\mathbb{C}$ of $F(f)$. An introduction to the properties of these sets was given in \cite{MR1216719}. 

For a general {\tef} the \itshape escaping set \normalfont $$I(f) = \{z : f^n(z)\rightarrow\infty\text{ as }n\rightarrow\infty\}$$ was studied first in \cite{MR1102727}. This paper concerns a subset of $I(f)$, the \itshape fast escaping set \normalfont $A(f)$. This was introduced in \cite{MR1684251}, and can be defined \cite{Rippon01102012} by
\begin{equation}
\label{Adef}
A(f) = \{z : \text{there exists } \ell \isnatural \text{ such that } |f^{n+\ell}(z)| \geq M^n(R,f), \text{ for } n \isnatural\}.
\end{equation}
Here the \itshape maximum modulus function \normalfont is defined by $M(r,f) = \max_{|z|=r} |f(z)|,$ for $r \geq 0.$ We write $M^n(r,f)$ to denote repeated iteration of $M(r,f)$ with respect to the variable $r$. In (\ref{Adef}), $R > 0$ is such that $M^n(R,f)\rightarrow\infty$ as $n\rightarrow\infty$.

A major open question in transcendental dynamics is the conjecture of Eremenko \cite{MR1102727} that, for every {\tef} $f$, $I(f)$ has no bounded components. A significant result regarding this conjecture was given by Rippon and Stallard, who showed \cite[Theorem~1.1]{Rippon01102012} that $A(f)$ has no bounded components. In view of this, the fast escaping set has been widely studied in recent years; see, for example, the papers \cite{Osborne1, MR2117213,MR2177186,MR2471088,MR2544754,rippon-2008, Rippon01102012, MR2838342} and \cite{Sixsmith3}. 

We introduce a partition of $A(f)$, and show that the two sets in this partition share many properties with $A(f)$. On the other hand, we show that there is a {\tef} such that the components of these sets have unexpected boundedness properties.

First we define 
\begin{equation}
\label{Adashdef}
A'(f) = \{z \in A(f) : \exists N\isnatural \text{ s.t. } |f^n(z)| = M(|f^{n-1}(z)|,f), \text{ for } n \geq N\},
\end{equation}
and we let $$A''(f) = A(f)\backslash A'(f).$$ We describe $A'(f)$ as the \itshape maximally fast escaping set\normalfont,  and $A''(f)$ as the \itshape non-maximally fast escaping set\normalfont. A set $S$ is \itshape completely invariant \normalfont if $z \in S$ implies that $f(z) \in S$ and $f^{-1}(\{z\}) \subset S$. It is clear from the definitions that both $A'(f)$ and $A''(f)$ are completely invariant.\\

Our first result shows that, in some sense, $A'(f)$ is at most a small set.
\begin{theo}
\label{Ttiny}
If $f$ is a {\tef}, then $A'(f)$ is contained in a countable union of curves each of which is analytic except possibly at its endpoints.
\end{theo}
In Example~\ref{Example1} we give a {\tef} such that $A'(f)$ has a single unbounded component, consisting of a countable union of analytic curves. It follows from Theorem~\ref{Ttiny} that this example is, in this sense, maximal. \\

It was shown in \cite{MR1684251} that it follows from the construction in \cite{MR1102727} that $A(f)~\ne~\emptyset$. In \cite[Remark 2]{MR2117213} it was further shown that $A(f) \cap J(f) \ne \emptyset$. In Example~\ref{Example1b} we give a {\tef} such that $A'(f)=\emptyset$. On the other hand, using a new annuli covering result and a recent result regarding the properties of the boundary of a multiply connected Fatou component, we show that $A''(f) \cap J(f)\ne\emptyset$; roughly speaking, this means that there are always points in the Julia set for which the rate of escape is "fast" but not "maximally fast".
\begin{theo}
\label{Tmain}
If $f$ is a {\tef}, then $A''(f) \cap J(f) \ne \emptyset.$
\end{theo}
It is known \cite[Theorem 5.1(c)]{Rippon01102012} that $A(f)$ is dense in $J(f)$, and \cite[Theorem~1.2]{Rippon01102012} that if $U$ is a Fatou component that meets $A(f)$, then $U$ is contained in $A(f)$. In the following theorem we strengthen these facts, and show that the sets $A'(f)$ and $A''(f)$ have strong dynamical properties relating to the Fatou and Julia sets. Note \cite[Theorem 4.4]{Rippon01102012} that all multiply connected Fatou components are in $A(f)$, and \cite{MRunknown, MR2959920} there exist {\tef}s with simply connected Fatou components contained in $A(f)$.
\begin{theo}
\label{Tprops}
Suppose that $f$ is a {\tef}. Then the following all hold.
\begin{enumerate}[(a)]
\item If $U$ is a simply connected Fatou component of $f$ such that $U \cap A(f) \ne \emptyset$, then $U \subset A''(f)$.
\item If $U$ is a multiply connected Fatou component of $f$, then $U \cap A''(f) \ne \emptyset$.
\item $A''(f)$ is dense in $J(f)$ and $J(f) \subset \partial A''(f)$.
\item If $A'(f) \cap J(f) \ne \emptyset$, then $A'(f)$ is dense in $J(f)$ and $J(f) \subset \partial A'(f)$. 
\item If $A'(f) \cap F(f) = \emptyset$, then $J(f) = \partial A''(f)$.
\item If $A'(f) \cap J(f) \ne \emptyset$ and $A'(f) \cap F(f) = \emptyset$, then $J(f) = \partial A'(f)$. 
\end{enumerate}
\end{theo}
Clearly it follows from part (a) that if $f$ has no multiply connected Fatou component, then $A'(f) \cap F(f) = \emptyset$. The additional hypothesis that $A'(f) \cap F(f) = \emptyset$ in parts (e) and (f) of Theorem~\ref{Tprops} is required; in Example~\ref{Emconn} we give a {\tef}, $f$, such that $$A'(f) \cap F(f) \ne \emptyset, \ J(f) \ne \partial A''(f) \text{ and } J(f) \ne \partial A'(f).$$ In Example~\ref{Emconna} we give a {\tef}, $f$, which has a multiply connected Fatou component and which satisfies $A'(f)=\emptyset$.\\

We recall the property, mentioned earlier, that $A(f)$ has no bounded components. By studying a function considered by Hardy we show that this property does not hold for $A'(f)$ or $A''(f)$. 
\begin{theo}
\label{Texample}
There is a {\tef} $f$ such that
\begin{enumerate}[(a)]
\item $A'(f)$ is uncountable and totally disconnected;
\item $A''(f)$ has uncountably many singleton components and at least one unbounded component.
\end{enumerate}
\end{theo}
On the other hand, in Example~\ref{Lexample} we give a {\tef} such that $A(f)$ has an unbounded component which is contained in $A'(f)$.  \\

The structure of this paper is as follows. To prove Theorem~\ref{Tmain} we require a new covering result for annuli, which is of independent interest. This is given in Section~\ref{S1}. In Section~\ref{S2} we prove Theorem~\ref{Ttiny} and Theorem~\ref{Tmain}. In Section~\ref{S3} we prove Theorem~\ref{Tprops}. Finally, in Section~\ref{Sfinal} we give the examples; the proof of Theorem~\ref{Texample} is included in Example~\ref{HExample}.
\section{A new covering result}
\label{S1}
In this section we give a new covering result for annuli, which is used to construct a point in $A''(f) \cap J(f)$ in the case that $f$ is a {\tef} with no multiply connected Fatou component. This result is similar to \cite[Theorem 2.2]{2013arXiv1301.1328R}, generalised to the whole family of iterates of a {\tef}. Here the \itshape minimum modulus function \normalfont is defined, for $r \geq 0$, by $m(r,f) = \min_{|z|=r} |f(z)|$. We also use the notation for an annulus $$A(r_1, r_2) = \{ z : r_1 < |z| < r_2 \}, \qfor 0 < r_1 < r_2, $$ and a disc $$B(\zeta, \ r) = \{ z : |z-\zeta| < r \}, \qfor 0 < r, \ \zeta\in\mathbb{C}.$$
\begin{theo}
\label{Tcovering}
Suppose that $f$ is a {\tef}, and that $$1<\lambda<\lambda'<\lambda''.$$ Then there exist $R' > 0$ and a function $\epsilon : \mathbb{R} \to \mathbb{R}$, both of which depend on $f,\lambda,\lambda'$ and $\lambda''$, such that the following holds. If $r \geq R'$, $n\isnatural$, $\eta \geq 0$ and
\begin{equation}
\label{funnyeq1}
\text{there exists } s \in (\lambda r, \lambda' r) \text{ such that } m(s, f^n) \leq \eta,
\end{equation}
then there exists $w \in B(0, M^n(r, f))$ such that
\begin{equation*}
f^n(A(r,\lambda'' r)) \supset B(0, M^n(r, f))\backslash B(w, \epsilon(r)\max\{|w|,\eta\}).
\end{equation*}
Moreover, $\epsilon(r)\rightarrow 0$ as $r\rightarrow\infty$.
\end{theo}
%
% Properties of M
%
To prove Theorem~\ref{Tcovering} we require some results regarding the maximum modulus of a {\tef}. The first two are well-known:
%\begin{equation}
%\label{Meq0}
%\log M(e^t,f) \text{ is a convex and increasing function of } t,
%M(r,f) \text{ is an increasing function of } r,
%\end{equation}
\begin{equation}
\label{Meq1}
\frac{\log M(r,f)}{\log r} \rightarrow\infty \text{ as } r\rightarrow\infty,
\end{equation}
and
\begin{equation}
\label{Meq2}
\frac{M(cr,f)}{M(r,f)} \rightarrow\infty \text{ as } r\rightarrow\infty, \qfor c>1.
\end{equation}

We require the following \cite[Lemma 2.2]{MR2544754}.%\cite[Theorem 2.2]{2011arXiv1109.1794B}.
\begin{lemma}
\label{LMtheo}
If $f$ is a {\tef}, then there exists $R_0~=~R_0(f)~>~0$ such that, for all $c>1,$% and all $n\isnatural$, 
\begin{equation*}
%\label{Mbigeq}
M(r^c, f) \geq M(r, f)^c, \qfor r\geq R_0.
\end{equation*}
\end{lemma}
We deduce the following by (\ref{Meq1}) and repeated application of Lemma~\ref{LMtheo}.
\begin{corollary}
\label{CMtheo}
If $f$ is a {\tef}, then there exists $R_1~=~R_1(f)~>~0$ such that, for all $c>1$ and all $n\isnatural$, 
\begin{equation*}
M^n(r^c, f) \geq M^n(r, f)^c, \qfor r\geq R_1.
\end{equation*}
\end{corollary}
Next, for a {\tef} $f$ and for $c>1$, we define a function
\begin{equation*}
\psi_c(r) = \frac{1}{2}\left(\inf_{n\isnatural}\frac{M^n(c r, f)}{M^n(r,f)} - 1\right),\qfor r > 0.
\end{equation*}
We need the following lemma regarding the function $\psi_c$.
\begin{lemma}
\label{Llambda}
If $f$ is a {\tef} and $c>1$, then $$\psi_c(r)\rightarrow\infty \text{ as } r\rightarrow\infty.$$
\end{lemma}
\begin{proof}
By Corollary~\ref{CMtheo} and by (\ref{Meq1}), there exists $R=R(f)>0$ such that, for all $n\isnatural$ and $c>1$, $$\frac{M^n(c r,f)}{M^n(r,f)} \geq \frac{M^n(r,f)^{1+\log c/\log r}}{M^n(r,f)} = M^n(r,f)^{\log c/\log r} \geq M(r,f)^{\log c/\log r}, \qfor r\geq R.$$ By a second application of (\ref{Meq1}) we see that $M(r,f)^{\log c/\log r}\rightarrow\infty$ as $r\rightarrow\infty$. The result follows.
\end{proof}
We also need the following, which is a version of part of \cite[Lemma 2.6]{Rippon01102012}.
\begin{lemma}
\label{Lother}
If $f$ is a {\tef} and $d>1$, then there exists $R_2=R_2(f,d)>0$ such that,  
\begin{equation*}
M(dr, f^n) \geq M^n(r, f), \qfor r \geq R_2, \ n\isnatural.
\end{equation*}
\end{lemma}
\begin{proof}
The proof of this result is identical to the proof of the corresponding part of \cite[Lemma 2.6]{Rippon01102012}, which gives the case that $d = 2$.
\end{proof}
%
% Hyperbolic metric
%
We also require some facts about the hyperbolic metric. If $G$ is a hyperbolic domain and $z_1, z_2 \in G$, then we denote the density of the hyperbolic metric in $G$ at $z_1$ by $\rho_G(z_1)$, and the hyperbolic distance between $z_1$ and $z_2$ in $G$ by $[z_1 ,z_2]_G$. 
By \cite[Theorem 9.13]{MR1049148}, there is an absolute constant $C > 1$ such that, with a suitable normalization of hyperbolic density, we have
\begin{equation}
\label{C1def}
\rho_{\mathbb{C}\backslash\{0,1\}}(z) \geq \frac{1}{2|z|\log(C|z|)}, \qfor z\in\mathbb{C}\backslash\{0,1\}.
\end{equation}
For each $\tau>1$, we define a constant $D_\tau>1$ such that
\begin{equation}
\label{C0def}
\frac{1}{2} \log D_\tau = \max_{\{z,z': |z|=|z'|=1\}} [z,z']_{A(1/\tau,\tau)}.
\end{equation}

We now prove Theorem~\ref{Tcovering}. 
\begin{proof}[Proof of Theorem~\ref{Tcovering}]
Suppose that $f$ is a {\tef} and that $1<\lambda<\lambda'<\lambda''.$ Set $$\tau = \min\left\{\lambda, \frac{\lambda''}{\lambda'}\right\} > 1 \quad\text{and}\quad c = \frac{1 + \lambda}{2}>1.$$ 

By Lemma~\ref{Llambda}, we can choose $R'$ sufficiently large that $\psi_{c}(r) > C^{D_\tau-1}$, for $r\geq R'$. We also assume that $R' \geq R_2$, where $R_2=R_2(f,d)$ is the constant from Lemma~\ref{Lother} with $d = \frac{\lambda}{c} > 1$.

Define the function $\epsilon: \mathbb{R}\to\mathbb{R}$ by
\begin{equation}
\label{epsilondef}
\epsilon(r) = \frac{2C^{1-1/D_\tau}}{\psi_{c}(r)^{1/D_\tau}},\qfor r \geq R'.
\end{equation}
Note that, since $C>1$ and $D_\tau >1$, 
\begin{equation}
\label{epqeq}
2/\epsilon(r) < \psi_{c}(r), \qfor r\geq R'.
\end{equation}
Note also, by Lemma~\ref{Llambda}, that $\epsilon(r)\rightarrow 0$ as $r\rightarrow\infty$.

Suppose that there exists $r\geq R'$, $n\isnatural$, $\eta\geq 0$ and $s \in (\lambda r, \lambda' r)$ such that $m(s, f^n) \leq \eta$. We choose $\zeta$ and $\zeta'$ such that $|\zeta|=|\zeta'|=s$, and, by Lemma~\ref{Lother},
\begin{equation}
\label{zetaeq}
|f^n(\zeta)| \leq \eta \text{ and } |f^n(\zeta')| = M(s,f^n) \geq M^n(cr,f).
\end{equation}
Set $A = A(s/\tau,s\tau) \subset A(r, \lambda'' r)$. Suppose, by way of contradiction, that $f^n|_{A}$ omits values $w_1, w_2 \in B(0, M^n(r,f))$ such that
\begin{equation}
\label{wdiff}
|w_2-w_1| = \beta \max\{|w_1|, \eta\} \text{ where } \beta \geq \epsilon(r).
\end{equation}
It follows by the contraction of the hyperbolic metric \cite[Theorem 4.1]{MR1230383} that
\begin{equation}
\label{C0ineq}
[\zeta,\zeta']_{A} \geq [f^n(\zeta),f^n(\zeta')]_{f^n(A)} > [\phi(f^n(\zeta)),\phi(f^n(\zeta'))]_{\mathbb{C}\backslash\{0,1\}},
\end{equation}
where $\phi(w) = (w-w_1)/(w_2-w_1)$. By (\ref{zetaeq}) and (\ref{wdiff}) we have $$|\phi(f^n(\zeta))| \leq \frac{|f^n(\zeta)|+|w_1|}{|w_2-w_1|} \leq \frac{\eta+|w_1|}{\beta\max\{|w_1|,\eta\}} \leq \frac{2}{\beta}$$ and $$|\phi(f^n(\zeta'))| \geq \frac{|f^n(\zeta')|-|w_1|}{|w_2|+|w_1|} \geq \frac{M^n(cr,f)-M^n(r,f)}{2M^n(r,f)} \geq \psi_{c}(r).$$

Note that, by (\ref{epqeq}) and (\ref{wdiff}), we have that $\psi_{c}(r) > 2/\beta$. Suppose first that $2/\beta \geq 1$. Then, by (\ref{C1def}), (\ref{C0def}) and (\ref{C0ineq}) we deduce that $$\frac{1}{2}\log D_\tau > \int_{2/\beta}^{\psi_{c}(r)} \frac{dt}{2t \log(Ct)} = \frac{1}{2}\log\frac{\log(C\psi_{c}(r))}{\log(2C/\beta)},$$ and hence $\beta < \epsilon(r)$, which is a contradiction. On the other hand, if $2/\beta < 1$, then we deduce similarly that $$\frac{1}{2}\log D_\tau > \int_{1}^{\psi_{c}(r)} \frac{dt}{2t \log(Ct)} = \frac{1}{2}\log\frac{\log(C\psi_{c}(r))}{\log C},$$ which is a contradiction to our choice of $R'$.
\end{proof}
In our proof of Theorem~\ref{Tmain}, we use the following immediate corollary of Theorem~\ref{Tcovering}. This is a generalisation of \cite[Corollary 2.3]{2013arXiv1301.1328R}, which considered the case $n=1$. Here we use the following notation for a closed annulus $$\overline{A}(r_1, r_2) = \{ z : r_1 \leq |z| \leq r_2 \}, \qfor 0 < r_1 < r_2.$$
\begin{corollary}
\label{Ccovering}
Suppose that $f$ is a {\tef}. Then there exists $R_3=R_3(f)>0$ such that the following holds. If there exists $r \geq R_3$, $n\isnatural$ and $s \in (2r, 4r)$ such that $m(s, f^n) \leq 1$, and $S, \ S', \ T, \ T'$ satisfy $$2 < S < S', \ T < T' < M^n(r, f) \text{ and } S' \leq \frac{1}{2}T,$$ then $f^n(A(r,8r))$ contains $\overline{A}(S,S')$ or $\overline{A}(T,T').$
\end{corollary}
\section{Properties of $A'(f)$ and $A''(f)$}
\label{S2}
In this section we prove Theorem~\ref{Ttiny} and Theorem~\ref{Tmain}.
\begin{proof}[Proof of Theorem~\ref{Ttiny}]
Adapting the notation in \cite{MR1662226}, let $\mathcal{M}(f)$ be the set of points where a {\tef}, $f$, attains its maximum modulus; that is,
\begin{equation}
\label{Mdef}
\mathcal{M}(f) = \{ z : |f(z)| = M(|z|,f) \}.
\end{equation} 
It is well-known -- see, for example, \cite[Theorem 10]{valironlectures} -- that $\mathcal{M}(f)$ consists of, at most, a countable union of \itshape maximal curves, \normalfont which are analytic except possibly at their endpoints. 

If $z \in A'(f)$, where $f$ is a {\tef}, then, by (\ref{Adashdef}), there exists $N\isnatural$ such that 
\begin{equation}
\label{Aeq}
f^n(z)\in\mathcal{M}(f), \qfor n\geq N.
\end{equation}
In particular, it follows that, 
\begin{equation*}
A'(f) \subset \bigcup_{k=0}^\infty f^{-k}(\mathcal{M}(f)).
\end{equation*}
Hence $A'(f)$ is contained in a countable union of curves, which are analytic except possibly at their endpoints, as required.
\end{proof}
%
% RS result on compact sets
%
In order to prove Theorem~\ref{Tmain} we require a number of preliminary results, the first of which is a version of \cite[Lemma 1]{rippon-2008}, which considered images of sets under a single iteration of $f$. 
\begin{lemma}
\label{LemmaRS}
Suppose that $(E_n)_{n\isnatural}$ is a sequence of compact sets and $(m_n)_{n\isnatural}$ is a sequence of integers. Suppose also that $f$ is a {\tef} such that $E_{n+1} \subset f^{m_n}(E_n )$, for $n\isnatural$. Set $p_n = \sum_{k=1}^n m_k$, for $n\isnatural$. Then there exists $\zeta\in E_1$ such that 
\begin{equation}
\label{feq}
f^{p_n}(\zeta) \in E_{n+1}, \qfor n\isnatural.
\end{equation}
If, in addition, $E_n \cap J(f) \ne \emptyset$, for $n\isnatural$, then there exists $\zeta \in E_1 \cap J(f)$ such that (\ref{feq}) holds.
\end{lemma}
\begin{proof}
For $n\isnatural$, we set $$F_n = \{z \in E_1 : f^{p_1}(z) \in E_2,  f^{p_2}(z) \in E_3, \ldots, f^{p_n} (z) \in E_{n+1} \}.$$ It follows by hypothesis that $(F_n)_{n\isnatural}$ is a decreasing sequence of non-empty compact subsets of $E_1$, and so $F = \bigcap_{k=1}^\infty F_k$ is a non-empty subset of $E_1$. We choose $\zeta \in F$ and the result follows.

Since $J(f)$ is completely invariant, the second statement follows by applying the first statement to the non-empty compact sets $E_n \cap J(f)$, for $n\isnatural$.
\end{proof}
%
% Bakers MCWD result
%
We require some results concerning multiply connected Fatou components. The first is the following well-known result of Baker \cite[Theorem 3.1]{MR759304}. We say that a set $U$ surrounds a set $V$ if $V$ is contained in a bounded component of $\mathbb{C}\backslash U$. If $U$ is a Fatou component, we write $U_n$, $n\geq 0$, for the Fatou component containing $f^n(U)$. We also let \rmfamily dist\normalfont$(z, U) = \inf_{w\in U} |z-w|$, for $z\iscomplex$.
\begin{lemma}
\label{Lbaker}
Suppose that $f$ is a {\tef} and that $U$ is a multiply connected Fatou component of $f$. Then each $U_n$ is bounded and multiply connected, $U_{n+1}$ surrounds $U_n$ for large $n$, and \upshape\rmfamily~dist\normalfont$(0, U_n)\rightarrow\infty$ as $n\rightarrow\infty$.
\end{lemma}
In addition we require \cite[Theorem 4.4]{Rippon01102012}.
\begin{lemma}
\label{mconnlemm}
Suppose that $f$ is a {\tef}, and that $U$ is a multiply connected Fatou component of $f$. Then $\overline{U}\subset A(f)$.
\end{lemma}
%
% The RA results
%
We also require some notation and results from \cite{Sixsmith3}. Define a function $R_A$ by
\begin{equation}
\label{RAdef}
R_A(z) =  \max \{R\geq 0 : M^n(R,f)\rightarrow\infty \text{ as } n\rightarrow\infty \text{ and } |f^n(z)| \geq M^n(R,f), \text{for } n\isnatural\},
\end{equation}
where we set $R_A(z) = -1$ if the set on the right-hand side of (\ref{RAdef}) is empty. If the set on the right-hand side of (\ref{RAdef}) is not empty, then the existence of a maximum follows from the continuity of $M(R,f)$.

If $U$ is a multiply connected Fatou component which surrounds the origin, then we define $\intb U$ as the boundary of the component of $\mathbb{C}\backslash\overline{U}$ that contains the origin. We use the following lemma, which is a combination of \cite[Lemma 4.2(c)]{Sixsmith3} and \cite[Lemma 4.4(a)]{Sixsmith3}. 
\begin{lemma}
\label{Lnhd}
Suppose that $f$ is a {\tef}. Then there exists $R_4=R_4(f)>0$ such that the following holds. Suppose that $U$ is a multiply connected Fatou component of $f$, which surrounds the origin and which satisfies \upshape\rmfamily~dist\itshape$(0,U)\geq R_4$. There there exists $R'\geq 0$ such that $$R_A(z) = R', \qfor z \in \intb U.$$
\end{lemma}

We also need the following simple result. 
\begin{lemma}
\label{LAdashandRA}
Suppose that $f$ is a {\tef}, that $z \in A'(f)$ and that $R_A(z) \geq 0$. Then there exists $N\isnatural$ such that $$|f^n(z)| = M^n(R_A(z),f), \qfor n\geq N.$$
\end{lemma}
\begin{proof}
Suppose that $z \in A'(f)$. It follows from (\ref{Adashdef}) that there exists $N\isnatural$ such that, with $R = |f^{N-1}(z)|$, we have  $M^n(R, f)\rightarrow\infty$ as $n\rightarrow\infty$, and $$|f^n(z)| = M^{n+1-N}(R, f), \qfor n\geq N.$$ Since $R_A(z) \geq 0$ we have, by (\ref{RAdef}), that $$M^{N-1}(0,f) \leq M^{N-1}(R_A(z),f) \leq |f^{N-1}(z)| = R.$$ It follows, by the continuity of the function $M$, that there exists $R' \geq R_A(z)$ such that  $M^{N-1}(R',f) = R$. We deduce that $$|f^n(z)| = M^{n}(R', f), \qfor n\geq N.$$ The result follows since, by (\ref{RAdef}), we must have $R_A(z) = R'$.
\end{proof}
We now prove Theorem~\ref{Tmain}. 
\begin{proof}[Proof of Theorem~\ref{Tmain}]
Let $f$ be a {\tef}. Our proof splits into two cases. If $f$ has a multiply connected Fatou component, then we show that there are, at most, countably many points in $A'(f) \cap J(f)$ on the inner boundary of some multiply connected Fatou component, which is a continuum in $A(f) \cap J(f)$.  If $f$ has no multiply connected Fatou component, then we use Theorem~\ref{Tcovering} to construct a point $\zeta$ which lies in $A''(f) \cap J(f)$. \\

Suppose first that $f$ has a multiply connected Fatou component, $U$. By Lemma~\ref{Lbaker} we may assume that $U$ surrounds the origin and that \upshape\rmfamily~dist\itshape$(0,~U) \geq R_4$\normalfont, where $R_4$ is the constant from Lemma~\ref{Lnhd}. It follows from Lemma~\ref{Lnhd} that there exists $R'\geq 0$ such that $R_A(z) = R'$, for $z \in \intb U.$ We deduce by Lemma~\ref{LAdashandRA} that for each $z \in \intb U \cap A'(f)$ there exists $N=N(z)\isnatural$ such that $|f^n(z)| = M^n(R',f)$, for $n\geq N.$

We recall the definition of the set $\mathcal{M}(f)$ in (\ref{Mdef}). It follows, by (\ref{Aeq}), that 
\begin{equation}
\label{imposseq}
\intb U \cap A'(f) \subset \bigcup_{k=0}^\infty f^{-k}\left(\{ z\in\mathcal{M}(f) : |z| = M^k(R',f)\}\right).
\end{equation}

It is known \cite[Theorem 10]{valironlectures} that, for each $R \geq 0$, the set $\{ z\in\mathcal{M}(f) : |z| = R \}$ is finite. Hence the right-hand side of (\ref{imposseq}) is countable. Since, by Lemma~\ref{mconnlemm}, $\intb U$ is a continuum contained in $A(f)\cap J(f)$, we deduce that $\intb U \cap A''(f) \cap J(f)$ is uncountable. \\

On the other hand, suppose that $f$ has no multiply connected Fatou component. Set $K = \frac{1}{2048}$ and let $R_3$ be the constant from Corollary~\ref{Ccovering}. (We note that this choice of $K$ is smaller than necessary for the proof of Theorem~\ref{Tmain}, but facilitates the proof of Lemma~\ref{L2}; see equations (\ref{quitesmall}) and (\ref{neededeq}) below.) Define a function $$\mu(r) = K M(r, f), \qfor r > 0.$$ 

We construct an increasing sequence of real numbers $(r_n)_{n\isnatural}$ and a sequence of integers $(m_n)_{n\isnatural}$. Choose $r_1 \geq R_3$ sufficiently large that $B(0, r_1) \cap J(f) \ne \emptyset$ and also, by (\ref{Meq1}), that 
\begin{equation}
\label{mugoesoff}
\mu(r) > \max\{r^2,  2\}, \qfor r \geq r_1.
\end{equation}
By (\ref{Meq2}), we may also assume that $r_1$ is sufficiently large that 
\begin{equation*}
\mu(r) \geq \frac{1}{K} M\left(Kr,f\right), \qfor r \geq r_1.
\end{equation*}
We deduce that
\begin{equation}
\label{mubig}
\mu^n(r) \geq M^n\left(Kr,f\right), \qfor r \geq r_1, \ n\isnatural.
\end{equation}

The construction proceeds inductively. Suppose that, for some $k\geq 1$, we have constructed the sequences $(r_n)_{n \leq k}$ and $(m_n)_{n < k}$. Let $A$ be the annulus $A = A(r_k, 8r_k)$ and let $A'$ be the annulus $A' = A(2r_k, 4r_k)$. Suppose that $$m(s, f^n) > 1, \qfor s \in (2r_k,4r_k), \ n\isnatural.$$ Then, by Montel's theorem, $\{f^n\}_{n\isnatural}$ is a normal family in $A'$, and so $A' \subset F(f)$. Hence, by the choice of $r_1$, $A'$ is contained in a multiply connected Fatou component of $f$, which a contradiction. Therefore, there exists $m_k \isnatural$ and $s \in (2r_k,4r_k)$ such that $m(s, f^{m_k}) \leq 1$.

Set $S=K M^{m_k}(r_k,f), \ S'=8S, \ T=16S$ and $T'=128S$. It follows from Corollary~\ref{Ccovering} that $f^{m_k}(A(r_k, 8r_k))$ contains either $\overline{A}(S, S')$, in which case we set $r_{k+1}~=~S$, or $\overline{A}(T, T')$, in which case we set $r_{k+1}=T$. Note that in either case we have, by (\ref{mugoesoff}), that $$r_{k+1} \geq K M^{m_k}(r_k,f) \geq \mu^{m_k}(r_k) > r_k.$$ This completes the construction of the sequences.

We now define a sequence of closed annuli $(E_n)_{n\isnatural}$ by $$E_n = \overline{A}(r_n, 8r_n), \qfor n\isnatural.$$ It follows at once from the above construction that $E_{n+1} \subset f^{m_n}(E_n)$, for $n\isnatural$. We also have, by the choice of $r_1$ and since $f$ has no multiply connected Fatou component,  that $E_{n}\cap J(f) \ne \emptyset$, for $n\isnatural$. Hence, by Lemma~\ref{LemmaRS}, there exists $\zeta\in J(f)$ such that (\ref{feq}) holds with $p_n = \sum_{k=1}^n m_k$, for $n\isnatural$.

We claim next that $\zeta \in A(f)$. We note, by (\ref{feq}) and by construction, that
\begin{align*}
|f^{p_n}(\zeta)| &\geq r_{n+1} \\
								 &\geq KM^{m_n}(r_n,f) \\
                 &\geq KM^{m_n}(KM^{m_{n-1}}(r_{n-1}, f),f) \\
                 &. \\
                 &. \\
                 &. \\
                 &\geq \mu^{p_n}(r_1), \qfor n\isnatural.
\end{align*}
By (\ref{mugoesoff}), there exists $\ell\isnatural$ such that $\mu^\ell(r_1) \geq r_1/K$. We deduce by (\ref{mubig}) that
\begin{equation}
\mu^{n+\ell}(r_1) \geq \mu^{n}\left(\frac{r_1}{K}\right) \geq M^n(r_1,f), \qfor n\isnatural.
\end{equation}

Hence $$|f^{p_n}(\zeta)| \geq \mu^{p_n}(r_1) \geq M^{p_n-\ell}(r_1,f), \qfor \text{sufficiently large values of } n.$$ It follows that $|f^k(\zeta)| \geq M^{k-\ell}(r_1,f)$, for sufficiently large values of $k$, and so  $\zeta \in A(f)$ as claimed.

Finally, suppose that $\zeta\in A'(f)$, in which case, by (\ref{Adashdef}), there exists $N\isnatural$ such that
\begin{equation}
\label{zetaneweq}
|f^{n+p}(\zeta)| = M^n(|f^p(\zeta)|,f), \qfor n\isnatural \text{ and } p\geq N.
\end{equation}

However, by the choices of $S$, $T$ and $K$, we have that $r_{n+1} \leq \frac{1}{128} M^{m_n}(r_n,f)$, for $n\isnatural$, and so if $z \in E_n$ and $f^{m_n}(z) \in E_{n+1}$, for some $n\isnatural$, then 
\begin{equation}
\label{zeq}
|f^{m_n}(z)| \leq 8r_{n+1} \leq \frac{1}{16} M^{m_n}(r_n,f) \leq \frac{1}{16} M^{m_n}(|z|,f).
\end{equation}
Hence, by (\ref{feq}) and (\ref{zeq}) we have, for all sufficiently large $n\isnatural$,
\begin{equation}
\label{quitesmall}
|f^{m_n+p_{n-1}}(\zeta)| = |f^{m_n}(f^{p_{n-1}}(\zeta))|
                         \leq \frac{1}{16} M^{m_n}(|f^{p_{n-1}}(\zeta)|,f).
\end{equation}

This is in contradiction to (\ref{zetaneweq}). We deduce that $\zeta \in A''(f) \cap J(f)$, as required.
\end{proof}
\section{Proof of Theorem~\ref{Tprops}}
\label{S3}
We require the following well-known distortion lemma; see, for example, \cite[Lemma 7]{MR1216719}.
\begin{lemma}
\label{dlemm}
Suppose that $f$ is a {\tef} and that $U \subset I(f)$ is a simply connected Fatou component of $f$. Suppose that $K$ is a compact subset of $U$. Then there exist $C > 1$ and $N\isnatural$
such that 
\begin{equation*}
\frac{1}{C}|f^n(z)| \leq|f^n(w)| \leq C|f^n(z)|, \qfor w, z \in K, n \geq N.
\end{equation*}
\end{lemma}
We also require the following \cite[Lemma 10]{rippon-2008}. Here a set $S$ is \itshape backwards invariant \normalfont if $z \in S$ implies that $f^{-1}(\{z\}) \subset S$.
\begin{lemma}
\label{Jlemm}
Suppose that $f$ is a {\tef}, and that $E~\subset~\mathbb{C}$ is non-empty.
\begin{enumerate}[(a)]
\item If $E$ is backwards invariant, contains at least three points and $\operatorname{int} E\cap J(f) = \emptyset$, then $J(f) \subset \partial E$.
\item If every component of $F(f)$ that meets $E$ is contained in $E$, then $\partial E \subset J(f)$.
\end{enumerate}
\end{lemma}
We now prove Theorem~\ref{Tprops}, and so we suppose that $f$ is a {\tef}. For part (a) of the theorem, let $U$ be a simply connected Fatou component of $f$ which meets $A(f)$. Note \cite[Theorem~1.2]{Rippon01102012} that $U \subset A(f)$. Suppose, by way of contradiction, that $z_0 \in U \cap A'(f)$. Then there exists $N\isnatural$, a point $z_N = f^N(z_0)$ and $R = |z_N|$ such that $$|f^n(z_N)| = M^n(R,f), \text{ for } n\isnatural.$$ 
Let $U_N$ be the Fatou component of $f$ containing $z_N$, and note (see, for example, \cite[Lemma 4.2]{MR3095155}) that $U_N$ is also simply connected. Choose a point $z_N' \in U_N$ such that $|z_N'| = R' < R$. Clearly $$|f^n(z_N')| \leq M^n(R',f), \text{ for } n\isnatural.$$ Since there exists $N'\isnatural$ such that $f^n(z_N') \ne 0$, for $n\geq N'$, we deduce that
\begin{equation}
\label{E1}
\frac{|f^n(z_N)|}{|f^n(z_N')|} \geq \frac{ M^n(R,f)}{ M^n(R',f)}, \qfor n\geq N'.
\end{equation}
Let $K=\{z_N,z_N'\}$. We deduce by Lemma~\ref{dlemm} that the left-hand side of (\ref{E1}) is bounded for $n\geq N'$. However, it follows from Corollary~\ref{CMtheo} that the right-hand side of (\ref{E1}) tends to infinity as $n\rightarrow\infty$. This contradiction completes the proof of part (a) of the theorem.

Suppose next that $U$ is a multiply connected Fatou component of $f$. Part (b) of the theorem is an immediate consequence of Lemma~\ref{mconnlemm} and Theorem~\ref{Ttiny}. 

For part (c) of the theorem, we note that if follows from Theorem~\ref{Tmain} that $A''(f) \cap J(f)$ is an infinite set, since for each $z \in A''(f)\cap J(f)$ at least one of the points $z, f(z)$ or $f^2(z)$ must have infinitely many preimages. It is known \cite[Theorem 4]{MR1216719} that the set of repelling periodic points of $f$ is dense in $J(f)$. Clearly $A''(f)$ contains no periodic points, and so $\operatorname{int} A''(f) \subset F(f)$. We deduce by Lemma~\ref{Jlemm}(a), applied with $E=A''(f) \cap J(f)$ and then with $E=A''(f)$, that $A''(f)$ is dense in $J(f)$, and that $J(f) \subset \partial A''(f)$.

For part (d) of the theorem, suppose that $A'(f) \cap J(f) \ne \emptyset$. For the same reasons as above, it follows that $A'(f) \cap J(f)$ is an infinite set and $\operatorname{int} A'(f) \subset F(f)$. We deduce by Lemma~\ref{Jlemm}(a), applied with $E=A'(f) \cap J(f)$ and then with $E=A'(f)$, that $A'(f)$ is dense in $J(f)$, and that $J(f) \subset \partial A'(f)$.

For part (e) of the theorem, suppose that $A'(f)\cap F(f) = \emptyset$. We deduce from part (a), Lemma~\ref{mconnlemm} and Lemma~\ref{Jlemm}(b) that $\partial A''(f) \subset J(f)$, and hence by part (c) that $J(f) = \partial A''(f)$.

Finally, for part (f) of the theorem, suppose that $A'(f) \cap J(f) \ne \emptyset$ and that $A'(f)\cap F(f) = \emptyset$. We deduce from Lemma~\ref{Jlemm}(b) that $\partial A'(f) \subset J(f)$, and hence by part (d) that $J(f) = \partial A'(f)$.
\section{Examples}
\label{Sfinal}
\begin{example}
\label{Example1}
\normalfont
Let $f_1(z) = \exp(z)$. We observe that $\mathcal{M}(f_1) = [0, \infty)$ and that $f_1(\mathcal{M}(f_1)) \subset \mathcal{M}(f_1)$. We deduce that $A'(f_1)$ is the countable union of analytic curves $$A'(f) = \bigcup_{k=0}^\infty f_1^{-k}([0, \infty)).$$ 

Since \cite[Lemma 4]{MR1216719} the set $\bigcup_{k=0}^\infty f_1^{-k}(-1)$ is dense in $J(f_1)$ and, as is well-known, $I(f_1)\subset J(f_1)$, we deduce that $\bigcup_{k=0}^\infty f_1^{-k}(-1)$ is dense in $A'(f_1)$. Thus $$\bigcup_{k=0}^\infty f_1^{-k}((-\infty,0))\subset A'(f_1) \subset \overline{\bigcup_{k=0}^\infty f_1^{-k}((-\infty,0))}.$$  Rempe \cite[Proposition 3.1]{MR2599894} showed that the set $\bigcup_{k=0}^\infty f_1^{-k}((-\infty,0))$ is connected. We deduce that $A'(f_1)$ is connected.
\end{example}
\begin{example}
\label{Example1b}
\normalfont
Let $f_2(z) = i\exp(z)$. We observe that $\mathcal{M}(f_2) = [0, \infty)$ and that $f_2(\mathcal{M}(f_2)) \cap \mathcal{M}(f_2) = \emptyset$. It follows that $A'(f_2) = \emptyset$.
\end{example}
\begin{example}
\label{Emconn}
\normalfont
Baker \cite{MR0153842} showed that constants $C>0$ and $0 < a_1 < a_2< \cdots$ can be chosen in such a way that $$f_3(z) = C z^2 \prod_{n=1}^\infty \left(1 + \frac{z}{a_n}\right)$$ is a {\tef} with a multiply connected Fatou component. Since $f_3$ only has positive coefficients in its power series, it is clear that $A'(f_3)$ contains an unbounded interval in the positive real axis. We deduce by Lemma~\ref{Lbaker} that $A'(f_3) \cap F(f_3) \ne \emptyset$. It follows, by Theorem~\ref{Tprops} part (b), that $\partial A''(f_3)$ and $\partial A'(f_3)$ both contain points of $F(f_3)$.
\end{example}
\begin{example}
\label{Emconna}
\normalfont
Let $f_4(z) = i f_3(z)$, where $f_3$ is the function in Example~\ref{Emconn}. It can be seen from the construction in \cite{MR0153842} that $f_4$ also has a multiply connected Fatou component. However, for the same reason as the function $f_2$ in Example~\ref{Example1b}, we see that $A'(f_4) = \emptyset$.
\end{example}
\begin{example}
\label{HExample}\normalfont
We define a family of {\tef}s by $$g_\alpha(z) =  \alpha\exp(e^{z^2} + \sin z), \qfor \alpha>0.$$ The function $g_1$ was considered by Hardy \cite{hardy1909}. It is clear that Theorem~\ref{Texample} follows from the following lemmas.
\begin{lemma}
\label{L1}
If $\alpha>0$, then $A'(g_\alpha)$ is uncountable and totally disconnected.
\end{lemma}
\begin{lemma}
\label{L2}
If $\alpha>0$ is sufficiently small, then there are uncountably many singleton components of $A''(g_\alpha)$. Moreover, $A''(g_\alpha)$ has at least one unbounded component.
\end{lemma}
\begin{proof}[Proof of Lemma~\ref{L1}]
To show that $A'(g_\alpha)$ is uncountable and totally disconnected we show that $A'(g_\alpha)$ consists of the preimages of a set $S \subset \mathbb{R}$, and we use properties of the maximum modulus of $g_\alpha$ to show that $S$ is uncountable and totally disconnected.

Although Hardy \cite{hardy1909} considered only $g_1$, it follows easily from his result that there exists $r_0 > 0$ such that the following holds. If $|z| \geq r_0$, then $z \in \mathcal{M}(g_\alpha)$ if and only if $z$ is real, and lies on the positive real axis if $\sin |z| > 0$ and on the negative real axis if $\sin |z| < 0$. Since $g_\alpha$ maps points on the real axis to the positive real axis, we deduce that 
\begin{equation}
\label{AandS}
A'(g_\alpha) = \bigcup_{k=0}^\infty g_\alpha^{-k}(S) \quad\text{where } S = \{x \geq r_0  : \sin g_\alpha^n(x) \geq 0, \text{ for } n \geq 0\}.
\end{equation} 

We prove first that $S$ is uncountable. Let $I_k = [2k\pi, (2k+1)\pi]$, for $k\isnatural$. We observe that $g'_\alpha(x)\rightarrow\infty$ as $x\rightarrow\infty$. Hence there exists $k_0 \geq r_0/2\pi$, $k_0\isnatural$, such that if $k\isnatural$ and $k\geq k_0$, then there exists $k' > k$ such that $I_{k'} \cup I_{k'+1} \subset g_\alpha(I_k)$.

We construct an increasing sequence of integers $(k_n)_{n\isnatural}$ as follows, noting that at each stage in the construction we have two distinct choices. First set either $k_1 = k_0$ or $k_1 = k_0+1$. Suppose inductively that $k_n$ is defined for some $n\isnatural$. There exists $k'_n > k_n$ such that $I_{k'_n} \cup I_{k'_n+1} \subset g_\alpha(I_{k_n})$. We choose either $k_{n+1} = k_n'$ or $k_{n+1} = k_n'+1$. This completes the construction of the sequence. We then find a point $\zeta \in S$ by Lemma~\ref{LemmaRS}, with $E_n = I_{k_n}$ and $m_n = 1$, for $n\isnatural$. 

Now let $S'$ be the subset of $S$ consisting of points which can be constructed in this way. Since at each stage in the construction we had two distinct choices, there is a surjection from $S'$ to the set of infinite binary strings. Hence $S'$, and so also $S$, is uncountable.

Next we prove that $S$ is totally disconnected. Let $J_k = [(2k+5/4)\pi, (2k+7/4)\pi]$, for $k\isnatural$. It follows from a construction similar to the above that if $k\geq r_0/2\pi$ is sufficiently large, then $J_{k}$ contains uncountably many points in $\mathbb{R}\backslash S$ and indeed in $\mathbb{R}\backslash A'(g_\alpha)$. (We simply replace all references to the intervals $I_n$ in the construction above to the intervals $J_n$, and thereby construct uncountably many points whose orbit lies in $\cup_{n\isnatural} \ J_n$, in which case each point lies in $\mathbb{R}\backslash A'(g_\alpha)$.)

Suppose next that $x_1 < x_2$ are points in $S$. Since $g_\alpha$ has a large derivative on the real axis, we deduce that there exists an arbitrarily large $k$, and $N\isnatural$ such that $g_\alpha^N(x_1) < (2k+5/4)\pi$ and $(2k+7/4)\pi < g_\alpha^N(x_2)$. Hence, there is a point $x' \in \mathbb{R}\backslash S$ such that $g_\alpha^N(x_1) < x' < g_\alpha^N(x_2)$. Since $g_\alpha$ is increasing on $\mathbb{R}$, there is a point $x''$ such that $x_1 < x'' < x_2$ and $g_\alpha^N(x'') = x'$. It follows that $x'' \notin S$, and we deduce that $S$ is totally disconnected, as required. 

It is known (\cite[Lemma 2.5]{MR3118409} and see also \cite[Chapter II]{MR0006493}) that a countable union of compact, totally disconnected subsets of $\mathbb{C}$ is totally disconnected. We deduce from this result, from the fact that $S$ is a countable union of compact, totally disconnected sets, and from (\ref{AandS}) that $A'(g_\alpha)$ is indeed totally disconnected.
\end{proof}

%\NEW To prove Lemma~\ref{L2} we need the following well-known result \cite[Lemma~2.1]{pre05533139}.
%\begin{lemma}
%\label{Lblow}
%Let $f$ be a transcendental entire function, let $K$ be a compact set with $K \cap E(f) = \emptyset$ and let $\Delta$ be a neighbourhood of $z \in J( f)$. Then there exists $N\isnatural$ such that $f^n(\Delta) \supset K$, for $n \geq N$. 
%\end{lemma}
%Here 
%\begin{equation*}
%E(f) = \{z : O^-(z) \text{ is finite}\}
%\end{equation*}
%and $$O^-(z) = \{w : f^n(w) = z, \text{ for some } n\isnatural\}.$$
%
%
%%%%%%%%%
%
%
\begin{proof}[Proof of Lemma~\ref{L2}] 
Consider first the real-valued function $x \to g_\alpha(x)$. It is straightforward to show that $g''_\alpha(x)>0$, for $x > 0$. We may assume, therefore, that $\alpha>0$ is sufficiently small that $g_\alpha(x)$ has exactly two fixed points. An elementary calculation shows that there is an attracting fixed point $p_\alpha \in (0,1)$, and a repelling fixed point $q_\alpha \in (p_\alpha, \infty)$.

Suppose that $x \in (q_\alpha, \infty)$. Then there exists $N=N(x)\isnatural$ such that, by the result of Hardy \cite{hardy1909} mentioned in the proof of Lemma~\ref{L1},
\begin{equation}
\label{gsize}
g_\alpha^{n+1}(x) \geq \frac{1}{e^2} \max\{g_\alpha(g_\alpha^{n}(x)),g_\alpha(-g_\alpha^{n}(x))\} = \frac{1}{e^2} M(g_\alpha^{n}(x), g_\alpha), \qfor n\geq N.
\end{equation}
It follows by \cite[Theorem 2.9]{Rippon01102012} that $x \in A(g_\alpha)$. Hence $(q_\alpha, \infty) \subset A(g_\alpha)$.

It was shown in the proof of Lemma~\ref{L1} that, for sufficiently large values of $k\isnatural$, $[(2k+5/4)\pi, (2k+7/4)\pi] \backslash A'(g_\alpha)$ is uncountable. We deduce that $(q_\alpha, \infty) \cap A''(g_\alpha)$ is uncountable. We claim that $(q_\alpha, \infty) \cap A''(g_\alpha)$ is totally disconnected.  Suppose that $x_1 < x_2$ are points in $(q_\alpha, \infty) \cap A''(g_\alpha)$. Arguing as in the proof of Lemma~\ref{L1}, we deduce that there exists an arbitrarily large $k$, and $N\isnatural$ such that $g_\alpha^N(x_1) < 2k\pi$ and $(2k+1)\pi < g_\alpha^N(x_2)$. It follows that there is a point $x' \in A'(g_\alpha)$ such that $g_\alpha^N(x_1) < x' < g_\alpha^N(x_2)$. Since $g_\alpha$ is increasing on $\mathbb{R}$, there is a point $x''$ such that $x_1 < x'' < x_2$ and $g_\alpha^N(x'') = x'$. It follows that $x'' \notin A''(g_\alpha)$, and we deduce that $(q_\alpha, \infty) \cap A''(g_\alpha)$ is totally disconnected, as claimed. \\

We now show that there exists $\alpha_0>0$ such that
\begin{equation}
\label{mainclaim}
(q_\alpha, \infty) \text{ is a component of } A(g_\alpha), \qfor 0 < \alpha < \alpha_0.
\end{equation}
Since $(q_\alpha, \infty) \cap A''(g_\alpha)$ is uncountable and totally disconnected, the fact that $A''(g_\alpha)$ has uncountably many singleton components, for $0 < \alpha < \alpha_0$, then follows from (\ref{mainclaim}). 

The proof of (\ref{mainclaim}) is complicated. Roughly speaking, our method is as follows. First we choose $\alpha$ sufficiently small that $F(g_\alpha)$ has an unbounded attracting basin $U$, which lies outside the escaping set and contains infinitely many preimages of the positive imaginary axis. We then deduce (\ref{mainclaim}) by a careful analysis of the properties of these preimages, and from the properties of the fast escaping set. \\
% We then deduce that there is an unbounded interval of the positive real axis which is a contained in $\partial U$. Next we deduce that this interval is a component of $A(g_\alpha)$. Finally, we deduce from this, and from the argument in Lemma~\ref{L1}, that $A''(g_\alpha)$ has uncountably many singleton components in this interval. \\

First we note some facts about the function $g_\alpha$. The proof is simplified slightly by noting that $g_\alpha$ satisfies the equation
\begin{equation}
\label{cconj}
g_\alpha(\overline{z}) = \overline{g_\alpha(z)}, \qfor z \in \mathbb{C}.
\end{equation}
We observe that if $z = x + iy$ and $g_\alpha(z) = Re^{i\theta} = u + iv$, then, by a calculation,
\begin{align}
\label{Req}
\log R &= e^{x^2-y^2}\cos 2xy + \sin x \cosh y + \log \alpha,\\
\label{Thetaeq} 
\theta &= e^{x^2-y^2}\sin 2xy + \cos x \sinh y, \\
\label{duxeq} 
%\frac{\partial u}{\partial x} &= ue^{x^2-y^2}(2x\cos 2xy - 2y\sin 2xy - \tan \theta(2x\sin 2xy + 2y\cos 2xy)) \\&+ u(\cos x \cosh y - \tan \theta \sin x \sinh y), \nonumber\\
\frac{\partial u}{\partial x} &= u(e^{x^2-y^2}(2x\cos 2xy - 2y\sin 2xy)+ \cos x \cosh y \\ &- \tan \theta(e^{x^2-y^2}(2x\sin 2xy + 2y\cos 2xy) + \sin x \sinh y)), \nonumber\\
\label{dvyeq} 
%\frac{\partial v}{\partial y} &= ve^{x^2-y^2}(-2y\cos 2xy - 2x\sin 2xy + \cot \theta(-2y\sin 2xy + 2x\cos 2xy)) \\&+ v(\sin x \sinh y + \cot \theta \cos x \cosh y).\nonumber\\
\frac{\partial v}{\partial y} &= v(e^{x^2-y^2}(-2y\cos 2xy - 2x\sin 2xy) + \sin x \sinh y \\&+ \cot\theta(e^{x^2-y^2}(-2y\sin 2xy + 2x\cos 2xy)+ \cos x \cosh y)).\nonumber
\end{align}
Let $\Gamma_0 = \{ z : \arg z = \pi/2\}$ and $\gamma_0 = \{ z : \arg z = \pi/4\}$. We deduce from (\ref{Req}) that $g_\alpha$ is bounded on $\Gamma_0$.

\begin{figure}[ht]
	\centering
	\includegraphics[width=12cm,height=9cm]{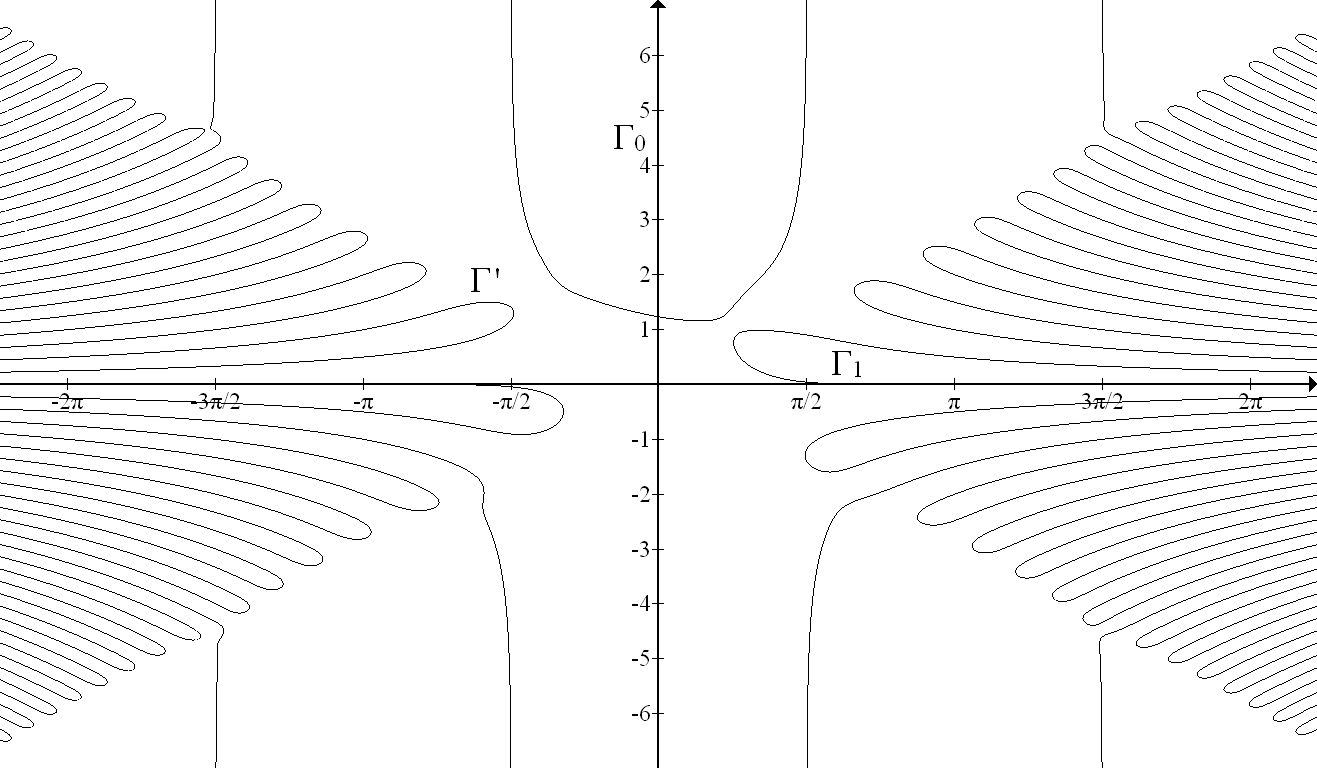}
	\caption{Some of the preimages of the positive imaginary axis, $\Gamma_0$, calculated by solving equation (\ref{Thetaeq}) for $\theta = \pi/2$.}
	\label{casen5}
\end{figure}

Next we define three domains $V'' \subset V' \subset V$ in which $g_\alpha$ has certain useful properties. First, for a large value of $K>1$, which we fix later, let $V$ be the domain $$V = \{ z=x+iy : x > K, \ 0 < y < e^{-x^2}\}.$$ It follows from (\ref{Req}) and (\ref{Thetaeq}) that we can choose $K$ sufficiently large that, for $z = x + iy \in V$,
\begin{equation}
\label{basicineq}
e^{x^2} - 2 \leq \log\frac{|g_\alpha(z)|}{\alpha} \leq e^{x^2} + 2 \quad\text{and}\quad 2y(xe^{x^2}-1) \leq \arg g_\alpha(z) \leq 2y(xe^{x^2}+1).
\end{equation}

We deduce that $V$ has unbounded intersection with a component of $g_\alpha^{-1}(\Gamma_0)$. Let $\Gamma_1$ be the intersection of this preimage component with $V$, and note by (\ref{basicineq}) that  % $\Gamma_1$ is asymptotic  \NEW as $x\rightarrow\infty$ \XNEW to the curve 
$$y \sim \frac{\pi}{4x}e^{-x^2} \text{ as } x\rightarrow\infty, \qfor x+iy \in \Gamma_1.$$ 

It follows by differentiating (\ref{Thetaeq}) that, increasing the size of $K$ if necessary, we may assume that
\begin{equation}
\label{Gamma1decreases}
\frac{dy}{dx} < 0, \qfor x+iy \in \Gamma_1.
\end{equation}

We deduce also that $V$ has unbounded intersection with a component of $g_\alpha^{-1}(\gamma_0)$. Let $\gamma_1$ be the intersection of this preimage component with $V$, and note by (\ref{basicineq}) that 
%$\gamma_1$ is asymptotic \NEW as $x\rightarrow\infty$ \XNEW to the curve $$y = \frac{\pi}{8x}e^{-x^2}.$$
$$y \sim \frac{\pi}{8x}e^{-x^2} \text{ as } x\rightarrow\infty, \qfor x+iy \in \gamma_1.$$

Increasing the size of $K$ if necessary, we may assume that $\Gamma_1$ and $\gamma_1$ each intersect the boundary of $V$ only at a point with real part $K$ and modulus less than $2K$. We may also assume that if $K<x<x'$, then
\begin{equation}
\label{nobigjump}
\text{if } x + iy \in \gamma_1, \text{ then there exists } x' + iy' \in \Gamma_1 \text{ such that } y' < 4y.
\end{equation}

Let $V'\subset V$ be the domain bounded by $\Gamma_1$, $\{ z : \operatorname{Re}(z) = K\}$, and the real axis. Increasing the size of $K$ again, if necessary, we may assume, by (\ref{basicineq}), that
\begin{equation}
\label{argrest}
0 < \arg g_\alpha(z) < 3\pi/4, \qfor z \in V'.
\end{equation}

Then let $V''\subset V'$ be the domain bounded by $\gamma_1$, $\{ z : \operatorname{Re}(z) = K\}$, and the real axis. We deduce from (\ref{duxeq}) and (\ref{dvyeq}) that if $x + iy \in V''\cup \gamma_1$ and $g_\alpha(x+iy) = u + iv$, then as $x\rightarrow\infty$, 
\begin{equation*}
\frac{\partial u}{\partial x} = u\left(2x e^{x^2} + O(x)\right) \quad\text{and}\quad \frac{\partial v}{\partial y} = v\left(2x e^{x^2} \cot(\arg g_\alpha(z)) + O(x)\right).
\end{equation*}

Hence, by (\ref{basicineq}), may assume that 
\begin{equation}
\label{partials}
\frac{\partial}{\partial x}\operatorname{Re}(g_\alpha(x+iy)) > 0 \quad\text{and}\quad \frac{\partial}{\partial y}\operatorname{Im}(g_\alpha(x+iy)) >0, \qfor x + iy \in V'' \cup \gamma_1.
\end{equation}

Finally, by a similar argument to the one given earlier relating to $\Gamma_1$, we note that there is a component $\Gamma'$ of $g_\alpha^{-1}(\Gamma_0)$ which is asymptotic to the negative real axis. We omit the details. Increasing $K$ one last time, if necessary, we may assume that $B(0,K)\cap \Gamma' \ne \emptyset$. Note that $\Gamma_1$, $\gamma_1$ and $\Gamma'$ are each independent of $\alpha$. \\

Next we fix a value of $\alpha_0$. We choose $0 < \alpha_0 < e^{-1}$ sufficiently small that if $0 < \alpha < \alpha_0$, then $g_\alpha$ has the two fixed points $p_\alpha$ and $q_\alpha$ discussed earlier, and also that $$g_\alpha \left(B(0, 2K) \cup \Gamma_0 \right) \subset B(0, 1).$$ 

Suppose then that $0 < \alpha < \alpha_0$. We deduce that $g_\alpha$ has an unbounded simply connected Fatou component, $U$, which contains $\Gamma_0$, the disc $B(0,2K)$ and so also $\Gamma_1$ and $\Gamma'$, the attracting fixed point $p_\alpha$, and indeed the interval $(0, q_\alpha)$. Note that $q_\alpha \geq 2K$, but we do not assume that equality holds. Clearly $U \cap A(g_\alpha) = \emptyset$, and so all preimages of $\Gamma_0$ lie outside $A(g_\alpha)$.

Clearly $q_\alpha \in J(g_\alpha) \cap \mathbb{C} \backslash A(g_\alpha)$. We claim that $(q_\alpha, \infty) \subset J(g_\alpha)$. Since $g_\alpha$ is bounded on $\Gamma_0$, it follows by Lemma~\ref{Lbaker} that $g_\alpha$ has no multiply connected Fatou component. Moreover, the set $A'(g_\alpha)$ is dense in $(q_\alpha, \infty)$. We deduce by Theorem~\ref{Tprops} part (a) that $(q_\alpha, \infty) \subset J(g_\alpha)$, as claimed.

Suppose next that $U' \ne U$ is a component of $F(g_\alpha)$ such that $g_\alpha(U') \subset U$. We claim that $U' \cap V' = \emptyset$. Suppose, to the contrary, that $U' \cap V' \ne \emptyset$. Since the boundary of $V'$ consists of points either in $U$ or in $J(g_\alpha)$, we deduce that $U' \subset V'$. Now \cite[Theorem 3]{MR1642181} $g_\alpha(U')$ and $U$ may differ by at most two points. However, $\Gamma' \subset U$ contains points with argument arbitrarily close to $\pi$. This is a contradiction, by (\ref{argrest}), completing the proof of our claim.

For $\xi > K$, let $y(\xi)$ be such that $\xi + iy(\xi)$ is the point on $\gamma_1$ of real part $\xi$; this point is unique by (\ref{partials}). By (\ref{basicineq}), we can choose $\xi > q_\alpha$ sufficiently large that
\begin{equation}
\label{bigthings}
\operatorname{Re}(g_\alpha(\xi+iy')) > 2\xi \quad\text{and}\quad \operatorname{Im}(g_\alpha(\xi+iy')) > 8y', \qfor y' \in (0,y(\xi)].
\end{equation}
We claim that if $y' > 0$, then there is a curve $\gamma' \subset U$ such that $$(\gamma' \cap \{ z: \operatorname{Re}(z)\geq\xi\}) \subset \{ z: 0< \operatorname{Im}(z)< y'\}.$$

\begin{figure}[ht]
	\centering
	\includegraphics[width=12cm,height=9cm]{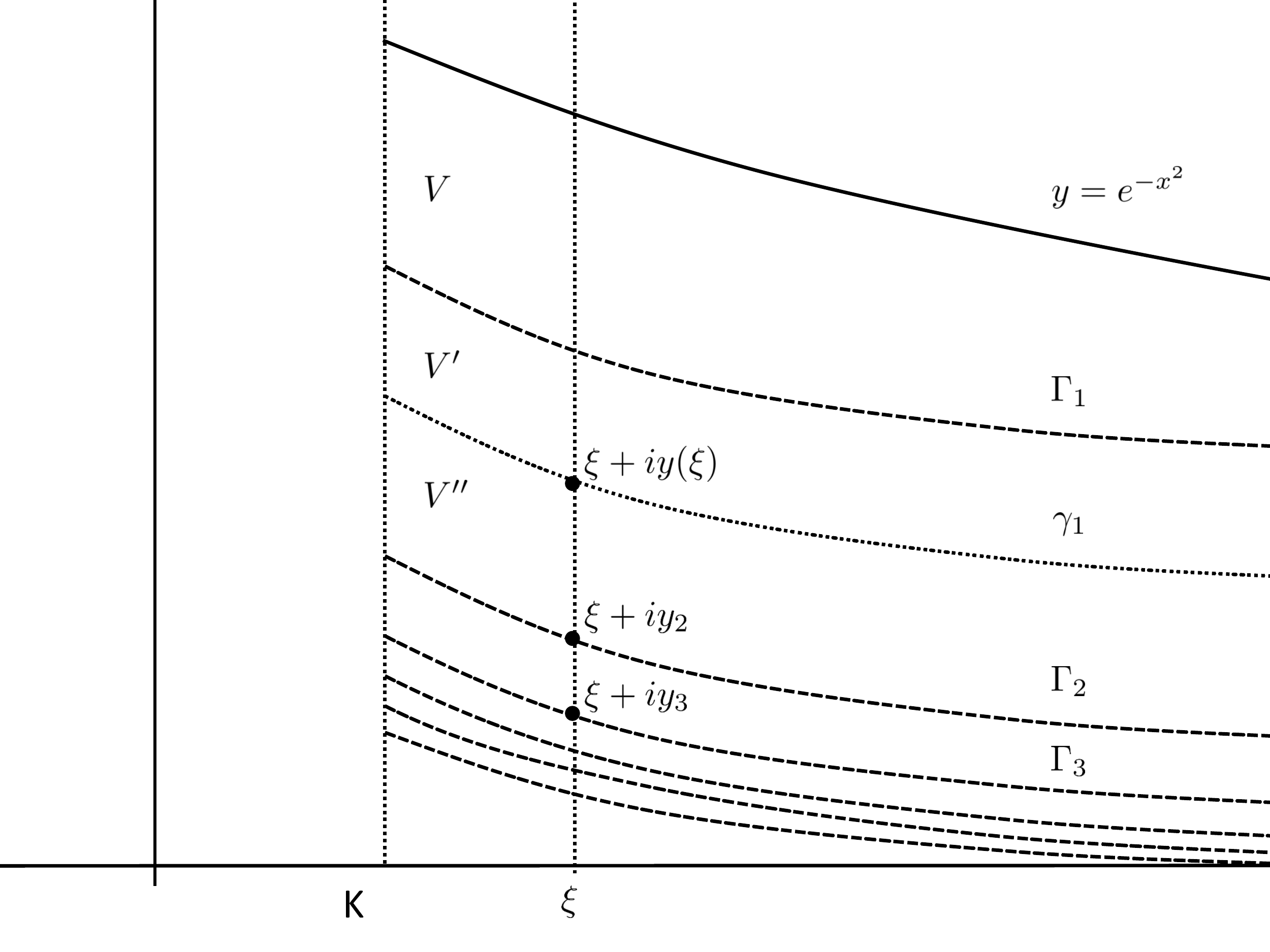}
	\caption{Some of the preimages of $\Gamma_0$, shown dashed, constructed in this proof.}% Note also that $V$ is the set bounded above by the solid line, $V'$ is the set bounded above by $\Gamma_1$, and $V''$ is the set bounded above by $\gamma_1$.}
	\label{f2}
\end{figure}

We construct sequences of disjoint curves $(\Gamma_n)_{n\geq 2}$ and positive real numbers $(y_n)_{n\geq 2}$ such that $\Gamma_n \subset U \cap V''$ is an unbounded curve with a finite endpoint on the line $\{z : \operatorname{Re}(z) = K\}$, for $n\geq 2$; see Figure~\ref{f2}.

We first define $\Gamma_2$ and $y_2$. It follows by (\ref{nobigjump}), (\ref{partials}) and (\ref{bigthings}) that there exists $0 < y_2 < y(\xi)$ such that $g_\alpha(\xi + iy_2) \in \Gamma_1 \subset U$. Let $\Gamma_{2}$ be the intersection of $V''$ with the component of $g_\alpha^{-1}(\Gamma_1)$ containing $\xi + iy_2$. By the comments above, $\Gamma_{2}$ cannot lie in a preimage component of $U$ distinct from $U$, and so must lie in $U$. Also $\Gamma_2$ is unbounded and has a finite endpoint on the line $\{z : \operatorname{Re}(z) = K\}$ by (\ref{partials}) and since it cannot intersect with $\gamma_1$ or with the real axis.

Suppose next that, inductively, we have defined $\Gamma_k$ and $y_k$, for some $k\geq 2$. By (\ref{partials}) and (\ref{bigthings}), there exist $0 < y_{k+1} < y_k$ such that $g_\alpha(\xi + iy_{k+1}) \in \Gamma_k \subset U$. Let $\Gamma_{k+1}$ be the intersection of $V''$ with the component of $g_\alpha^{-1}(\Gamma_k)$ containing $\xi + iy_{k+1}$. We have that $\Gamma_{k+1}\subset U$ for the same reasons that applied to $\Gamma_2$. Finally $\Gamma_{k+1}$ is unbounded and has a finite endpoint on the line $\{z : \operatorname{Re}(z) = K\}$ by (\ref{partials}) and since it cannot intersect with $\Gamma_k$ or with the real axis. This completes the construction of the sequences. 

Suppose that $n\geq 2$. It follows by (\ref{Gamma1decreases}) and (\ref{partials}) that $$y_n = \max \{\operatorname{Im}(z): z \in \Gamma_n \text{ and } \operatorname{Re}(z)\geq\xi\}.$$ It also follows by (\ref{partials}) and (\ref{bigthings}) that $y_{n+1} < y_n/8$, for $n\isnatural$. This completes the proof of the claim following equation (\ref{bigthings}).

%Next we claim that $[q_\alpha,\infty)\subset \partial U$. Let $W$ be the component of $\mathbb{C}\backslash U$ containing $[q_\alpha, \infty)$, and let $W'=W\backslash [\xi, \infty)$. Note that if $|z| > 3K$ and $z \notin [\xi,\infty)$, then $z$ can be separated from $[\xi,\infty)$ by a curve in $U$ formed from $\Gamma_n$, for some $n$ is natural, the conjugate of $\Gamma_n$ (by (\ref{cconj})), and a curve in $U \cap B(0, 2K)$. We deduce that $W'$ is bounded. Suppose that $x \in [q_\alpha, \xi)$ and $x \notin \partial U$. Let $\Delta$ be a neighbourhood of $x$ sufficiently small that $\Delta \cap U = \emptyset$ and indeed that $\Delta\subset W'$. By Lemma~\ref{Lblow} there exist $z \in \Delta$ and $n\isnatural$ such that $g_\alpha^n(z) \in U$. Let $U'$ be the Fatou component containing $z$. Clearly $U'$ is unbounded, and so $U' = U$. Hence $\Delta$ intersects with $U$ which is a contradiction. This establishes our claim.
Let $W$ be the component of $\mathbb{C}\backslash U$ containing $[q_\alpha, \infty)$, and let $W'=W\backslash [\xi, \infty)$. We claim that $W'$ is bounded. To prove this we let $\Gamma_1'$ be the curve in $U$ formed by the union of $\Gamma_1$, the complex conjugate of $\Gamma_1$ (recall (\ref{cconj})), and the part of the line $\{ z: \operatorname{Re}(z)=K\}$ joining the finite endpoints of these curves. Let $S$ be the bounded component of $\mathbb{C}\backslash (\Gamma_1' \cup\{z : \operatorname{Re}(z) = \xi\})$. We claim that, in fact, $W' \subset S$. For, suppose that $z \in W' \cap (\mathbb{C}\backslash S)$. If $z$ lies in the component of $\mathbb{C}\backslash\Gamma_1'$ which does not contain $[q_\alpha, \infty)$, then $\Gamma_1'\subset U$ separates $z$ from $[q_\alpha, \infty)$. On the other hand, if $z$ lies in the component of $\mathbb{C}\backslash\Gamma_1'$ which does contain $[q_\alpha, \infty)$, then Re$(z)\geq \xi$ and Im$(z)\ne 0$. Recalling the sequence of disjoint curves $(\Gamma_n)_{n\geq 2}$ constructed earlier, we see that  $z$ can be separated from $[q_\alpha, \infty)$ by the curve in $U$ formed by the union of $\Gamma_n$, for some $n\isnatural$, the complex conjugate of $\Gamma_n$, and the part of the line $\{ z: \operatorname{Re}(z)=K\}$ joining the finite endpoints of these curves. Hence $W'$ is bounded as claimed.

Let $T$ be the component of $A(g_\alpha)$ containing $(q_\alpha, \infty)$. Suppose, contrary to (\ref{mainclaim}), that $T \ne (q_\alpha, \infty)$. Then, since $q_\alpha \notin A(g_\alpha)$, there exists $\zeta \in T$ such that $\operatorname{Im}(\zeta) \ne 0$. We have that $\zeta \in W'$ as otherwise $U$ separates $\zeta$ from $(q_\alpha, \infty)$.

It follows from (\ref{Adef}) that there exist $\ell \in\mathbb{Z}$ and $R > 0$ such that $M^n(R,g_\alpha)\rightarrow\infty$ as $n\rightarrow\infty$ and $\zeta \in A_R^\ell(g_\alpha)$, where
\begin{equation}
\label{ARdef} A_R^\ell(g_\alpha) = \{ z : |g_\alpha^{n}(z)| \geq M^{n+\ell}(R,g_\alpha), \text{ for } n \isnatural, \ n+\ell\isnatural\} \subset A(g_\alpha).
\end{equation}

Let $T'$ be the component of $A_R^\ell(g_\alpha)$ containing $\zeta$. We have \cite[Theorem~1.1]{Rippon01102012} that $T'$ is closed and unbounded. Since $T'\cap U = \emptyset$ and $W'$ is bounded,  %$T'\cap [\xi, \infty)\ne\emptyset$, and 
%\NEW indeed \XNEW
%hence, by (\ref{ARdef}), 
$T'$ 
%has unbounded intersection with 
contains $[\xi, \infty)$. We deduce also that $g_\alpha(T') \subset T'$.\

Let $\xi'>q_\alpha$ be the smallest value such that $[\xi',\infty)\subset T'$. Since $W'$ is bounded,  by (\ref{ARdef}) there exists $N\isnatural$ such that $$g_\alpha^N(T') \cap W'=\emptyset.$$ Since $T'$ is connected and $g_\alpha^N(T')\subset T'$, we deduce that $g_\alpha^N(T') \subset [\xi, \infty)$. Hence $T'$ contains a curve $T''$ such that $\{\zeta\} \cup 
%has unbounded intersection with 
[\xi', \infty)\subset T''$.

Choose $\xi''$ such that $q_\alpha < \xi'' < \xi'$. Set $\widehat{T} = T'' \cup [\xi'',\xi')$. We note that there exists $N'\isnatural$ such that $g_\alpha^{N'}(\widehat{T}) \subset [\xi, \infty)$. Hence there is no neighbourhood of $\xi'$ in which $g_\alpha^{N'}$ is a homeomorphism. This is a contradiction, since $g_\alpha'(x) > 0$, for $x>0$. This contradiction completes our proof of (\ref{mainclaim}). Hence, as already observed, $(q_\alpha, \infty)$ contains uncountably many singleton components of $A''(g_\alpha)$. \\

%This is contradictory to the facts that $[q_\alpha, \infty) \subset \partial U$ and $T' \cap U = \emptyset$. Hence $T = (q_\alpha, \infty)$, as claimed. \\
%
%Suppose that $x \in (q_\alpha, \infty)$. There exists $N'\isnatural$ such that $g_\alpha^{N'}(x) > \xi$. If $\Delta$ is a sufficiently small neighbourhood of $x$ then $g_\alpha^{N'}(\Delta) = \Delta'$ is a small neightbourhood of $g_\alpha^{N'}(x)$. By the properties of $V$ there exists a series of cross-cuts of $\Delta'$ which tend to the real line and lie outside $A(g_\alpha)$. We deduce that there exists a series of cross-cuts of $\Delta$ which tend to the real line and lie outside $A(g_\alpha)$. This is contradictory to the properties of $T'$ mentioned at the end of the last paragraph. Hence $T = (q_\alpha, \infty)$, as claimed. \\
%
%\XNEW By a very similar argument to that relating to the set $S$ in the proof of Lemma~\ref{L1}, we may deduce that $A''(g_\alpha) \cap (q_\alpha, \infty)$ is uncountable and totally disconnected. Hence $A''(g_\alpha)$ has uncountably many singleton components which are contained in the interval $(q_\alpha, \infty)$.\\

Finally we show that $A''(g_\alpha)$ has at least one unbounded component. We recall that $g_\alpha$ has no multiply connected Fatou component. Hence we may suppose that $\zeta \in A''(g_\alpha)$ is the point constructed in the proof of Theorem~\ref{Tmain}, with $f = g_\alpha$. We note that since $|g_\alpha(z)|<1$, for $z \in \Gamma_0$, we may assume that the sequence $(m_k)_{k\isnatural}$, also constructed in the proof of Theorem~\ref{Tmain}, satisfies $m_k = 1$, for $k\isnatural$. 

Let $Q$ be the component of $A''(g_\alpha)$ containing $\zeta$, and suppose that $Q$ is bounded. Let $Q'$ be the component of $A(g_\alpha)$ containing $\zeta$. Since \cite[Theorem~1.1]{Rippon01102012} all components of $A(g_\alpha)$ are unbounded, $Q' \ne Q$. In particular $Q'$ contains a point $\zeta' \in A'(g_\alpha)$. 

Since $\zeta' \in A'(g_\alpha)$, there exists $N\isnatural$ such that $g_\alpha^N(\zeta')\in(q_\alpha, \infty)$. Since $(q_\alpha, \infty)$ is a component of $A(g_\alpha)$,  we obtain that $g_\alpha^N(Q')\subset (q_\alpha, \infty)$. In particular we have that $g_\alpha^N(\zeta)\in(q_\alpha, \infty)$. 

%We observe that $g_\alpha(x) \sim M(x, g_\alpha)$ as $x\rightarrow\infty$. 
We deduce, by (\ref{gsize}) with $x = g_\alpha^N(\zeta)$, that there exists $N'\isnatural$ such that 
\begin{equation}
\label{neededeq}
g_\alpha(g^n_\alpha(\zeta)) \geq \frac{1}{e^2}M(g_\alpha^n(\zeta), g_\alpha), \qfor n\geq N'.
\end{equation}
This is a contradiction to (\ref{quitesmall}), and so $Q$ is unbounded as required.
\end{proof}
\end{example}
\begin{example}
\label{Lexample}\normalfont
We define a family of {\tef}s by $$h_\alpha(z) =  \alpha e^z, \qfor \alpha\in (0, e^{-1}).$$ This family is contained in the \itshape exponential family\normalfont, the dynamics of which have been widely studied. It is well-known that $h_\alpha$ has an unbounded simply connected Fatou component, which contains the imaginary axis and an attracting fixed point, and also a repelling fixed point $q > 1$.

As in Example~\ref{Example1}, we see that $(q, \infty) \subset A'(h_\alpha)$. The techniques of the proof of Lemma~\ref{L2} may be used to show that $(q, \infty)$ is a component of $A(h_\alpha)$. We omit the details.
\end{example}
%
% Ack
%
\itshape Acknowledgment: \normalfont
The author is grateful to Gwyneth Stallard and Phil Rippon for all their help with this paper, and also John Osborne for his drawing his attention to \cite{MR2599894}.
%%%%%%%%%%%%%
%
% BIBLIOGRAPHY
%
%%%%%%%%%%%%%
\bibliographystyle{acm}
\bibliography{../../Research.References}
\end{document}